\documentclass{amsart}

\usepackage{amssymb,amsmath,amsthm}
\usepackage{graphicx}
\usepackage[cmtip,all]{xy}

\newtheorem{theorem}{Theorem}
\newtheorem{lemma}{Lemma}[section]
\newtheorem{prop}[lemma]{Proposition}
\newtheorem{cor}[theorem]{Corollary}

\theoremstyle{definition}

\theoremstyle{remark}
\newtheorem{remark}[lemma]{Remark}

\newcommand{\abs}[1]{\left\lvert#1\right\rvert}
\newcommand{\norm}[1]{\left\lVert#1\right\rVert}

\newcommand{\CB}{\mathcal{B}}

\newcommand{\CI}{\mathcal{I}}
\newcommand{\CJ}{\mathcal{J}}

\newcommand{\CM}{\mathcal{M}}

\newcommand{\vol}{{\rm vol}}

\begin{document}

\title[Points on curves in some small rectangles]
{The distribution of points on curves over finite fields in some small rectangles}

%    author one information
\author{Kit-Ho Mak}
\address{School of Mathematics \\
Georgia Institute of Technology \\
686 Cherry Street \\
Atlanta, GA 30332-0160, USA}
\email{kmak6@math.gatech.edu}

\subjclass[2010]{Primary 11G25; Secondary 11K36, 11T99}
\keywords{almost all, patterns, curves, rational points, uniform distribution}

%\thanks{The second author is supported by NSF grant number DMS - 0901621.}

\begin{abstract}
Let $p$ be a prime. We study the distribution of points on a class of curves $C$ over $\mathbb{F}_p$ inside very small rectangles $\CB$ for which the Weil bound fails to give nontrivial information. In particular, we show that the distribution of points on $C$ over some long rectangles is Gaussian.
\end{abstract}

\maketitle

\section{Introduction and statements of results}

Let $p$ be a large prime, and let $C\subseteq\mathbb{A}^2_p:=\mathbb{A}^2(\mathbb{F}_p)$ be an absolutely irreducible affine plane curve over $\mathbb{F}_p$ of degree $d>1$. We identify the affine plane with the set of points with integer coordinates in the square
$[0,p-1]^2$. For a rectangle $\CB=\CI\times\CJ \subseteq [0,p-1]^2$, we define $N_{\CB}(C)$ to be the number of (rational) points on $C$ inside $\CB$. When $\CB=[0,p-1]^2$, we will write $N(C)=N_{[0,p-1]^2}(C)$ for the number of points on $C$. It is widely believed that the points on $C$ are uniformly distributed in the plane. That is,
\begin{equation}\label{eqnNBud}
N_{\CB}(C)\sim N(C)\cdot\frac{\vol(\CB)}{p^2}.
\end{equation}
In fact, using some standard techniques involving exponential sums, one can show that the classical Weil bound \cite{Wei48} together with the Bombieri estimate \cite{Bom66} imply
%\begin{equation}\label{eqnLWfull}
%N(C) = p+O(\sqrt{p})
%\end{equation}
\begin{equation}\label{eqnLW}
N_{\CB}(C) = N(C)\cdot\frac{\vol(\CB)}{p^{2}}+O(d^2\sqrt{p}\log^2{p}),
\end{equation}
where the implied constant is absolute. If $f$ and $g$ are two functions of $p$, we write
\begin{equation}\label{eqnomega}
f=\Omega(g)
\end{equation}
to denote the function $f/g$ tends to infinity as $p$ tends to infinity. In other words, \eqref{eqnomega} is equivalent to $g=o(f)$.
The main term of \eqref{eqnLW} dominates the error term when
\begin{equation}\label{eqnvolBlarge}
\vol(\CB)\gg p^{\frac{3}{2}}\log^{2+\epsilon}{p}.
\end{equation}
In those cases \eqref{eqnNBud} holds. A natural and intriguing  question that arises is whether \eqref{eqnNBud} continues to hold for smaller boxes $\CB$. However, very few is known for the number of points $N_{\CB}(C)$ in a small $\CB$. Indeed, given a particular small $\CB$ that do not satisfy \eqref{eqnvolBlarge}, we do not even know if $\CB$ contains a point or not.

%% Currently, very few is known for such small $\CB$. For some special curves, there are non-trivial upper bounds for $N_{\CB}(C)$ that hold for smaller $\CB$ when $\CB$ is a square (i.e. when $\abs{\CI}=\abs{\CJ}$). These include the cases of the modular hyperbola \cite{ChSh10}, exponential curves \cite{ChSh10}, quadratic forms \cite{Zum11}, the graph of a polynomial \cite{CGOS12}, and hyperelliptic curves \cite{CCGHSZ11}. However, there are no known \textit{lower} bounds of any type. Indeed, given a particular small $\CB$ that do not satisfy \eqref{eqnvolBlarge}, we do not even know if $\CB$ contains a point or not.

One way to study $N_{\CB}(C)$ for small $\CB$ is to consider results on average. For instance, Chan \cite{Cha11} considered the number of points on average on the modular hyperbola $xy\equiv c$ modulo an odd number $q$, and showed that almost all (here ``almost all'' means with probability one) boxes satisfying
\begin{equation*}
\vol(\CB)\gg O(q^{\frac{1}{2}+\varepsilon})
\end{equation*}
have the expected number of points. Recently, Zaharescu and the author \cite{MaZa13} generalized the result of Chan to all curves over $\mathbb{F}_p$.

Another result of similar sort with only one moving side for $C$ being the modular hyperbola was obtained by Gonek, Krishnaswami and Sondhi \cite{GKS02}. In our language, they showed that if $\CB=(x,x+H]\times\CJ$ with $H$ very small and $\CJ$ of size comparable to $p$, then the numbers of points inside $\CB$ exhibit a Gaussian distribution when we move the box $\CB$ horizontally. A Gaussian distribution is also obtained by Zaharescu and the author \cite{MaZa11} in a similar situation. More precisely, we show that under some natural conditions, for $C$, $\CB$ as above, and if at least one of the character $\chi$, $\psi$ is nontrivial, the projections of the values of the hybrid exponential sum
\begin{equation}\label{eqnhyexp}
S = \sum_{P_i\in C\cap\CB}\chi(g(P_i))\psi(f(P_i))
\end{equation}
to any straight lines passing through the origin exhibit a Gaussian distribution when we move $\CB$ horizontally. We note that when $C$ is the affine line, a two-dimensional distribution of $S$ is obtained by Lamzouri \cite{Lam11}.

The aim of this paper is continue the study of $N_{\CB}(C)$ for small rectangle $\CB$. In particular, we show that for a large class of curves $C$, the distribution of $N_{\CB}(C)$ for the $\CB$ above is Gaussian. Our first step is to study the \textit{patterns} of points on curves, which is crucial for our study of $N_{\CB}(C)$ and may be of independent interest.

The study of patterns was first introduced by Cobeli, Gonek and Zaharescu \cite{CGZ03}, where they get results for the distribution of patterns of multiplicative inverses modulo $p$. We generalize their definition of patterns to curves by viewing the patterns in \cite{CGZ03} as patterns on the two coordinates for the curve $xy=1$. For any positive integer $s$, let $\mathbf{a}=(a_1,\ldots,a_s), \mathbf{b}=(b_1,\ldots,b_s)$ be two vectors so that all $a_i$'s are coprime to $p$, and all $a_1^{-1}b_1, \ldots, a_s^{-1}b_s$ are distinct modulo $p$. Define an $(\mathbf{a},\mathbf{b})$-\textit{pattern} to be an $s$-tuple of points $(P_1,\ldots,P_s)$, where each $P_i$ is of the form $(a_ix+b_i,y_i)$ for some $x$. As in \cite{CGZ03}, we may further restrict all the $y_i$ to lie in a specific interval $\CJ$ as we see fit.

In the case of the modular hyperbola in \cite{CGZ03}, if $\CJ=[0,p-1]$ the number of patterns is just $p-1$, since each $x$ corresponds to exactly one $y$ on the curve. However, for a general curve $C$ and any vector $\mathbf{a}, \mathbf{b}$, we do not know \textit{a priori} that even one pattern exists, since the two coordinates will not in general corresponds bijectively. Nevertheless, we are able to estimate the number of patterns for a large class of curves. Let $P(\CI,\CJ):=P_{\mathbf{a},\mathbf{b}}(C;\CI,\CJ)$ be the number of patterns with $x\in \CI$ and all $y$-coordinates lie in $\CJ$, then we have the following.

\begin{theorem}\label{thm1}
Let $C$ be a plane curve given by the equation $f(x,y)=0$. Let
\begin{equation}\label{eqnpi}
\pi:C\rightarrow \mathbb{A}^1,~(x,y)\mapsto x
\end{equation}
be the projection of $C$ to the first coordinates over $\overline{\mathbb{F}}_p$. Suppose there is an $x\in\overline{\mathbb{F}}_p$ so that $\pi$ ramifies completely, and let $\mathbf{a}=(a_1,\ldots,a_s), \mathbf{b}=(b_1,\ldots,b_s)$ be two vectors so that all $a_i$'s are coprime to $p$, and all $a_1^{-1}b_1, \ldots, a_s^{-1}b_s$ are distinct modulo $p$, then
\begin{equation*}
P(\CI,\CJ) = \abs{\CI}\left(\frac{\abs{\CJ}}{p}\right)^s+O(d^{2s}\sqrt{p}\log^{s+1}p).
\end{equation*}
In the case $\CI=[0,p-1]$, the error term can be slightly improved.
\begin{equation*}
P([0,p-1],\CJ) = p\left(\frac{\abs{\CJ}}{p}\right)^s+O(d^{2s}\sqrt{p}\log^{s}p).
\end{equation*}
\end{theorem}
Note that our estimation for the number of patterns is independent of $\mathbf{a}$ and $\mathbf{b}$.

% To understand the distribution of patterns in short intervals, we fix a length $H$ and set $P(x,H)=P((x,x+H],\CJ)$. Each pattern of $C$ is counted in exactly $H$ intervals of the form $(x,x+H]$, so by Theorem \ref{thm1}, the average of $P(x,H)$ is
%\begin{equation*}
%\frac{1}{p}\sum_{x=0}^{p-1}P(x,H)=\frac{HP([0,p-1],\CJ)}{p}.
%\end{equation*}
%Our next theorem estimates the second moment of $P(x)$ about its mean.
%\begin{theorem}\label{thm2}
%Let $C$, $\mathbf{a}$, $\mathbf{b}$ be as in Theorem \ref{thm1}, and let $H$, $P(x,H)$ be as above. Define
%% TBD
%\end{theorem}

% An immediate corollary is a lower bound for the length of $\CJ$ in order to guarantee $\CB_x=(x,x+H]\times\CJ$ contains a pattern for almost all $x$. %% "Almost all" to be confirmed.
%\begin{cor}\label{cor2}
%Let $C, \mathbf{a}, \mathbf{b}$ be as in Theorem \ref{thm1}. Then $P(x,H)>0$ for almost all %% "Almost all" to be confirmed.
%$x$ if $HN^s\gg p^s$ and $N\gg p^{1-\frac{1}{4s}}\log{p}$. %% Condition to be confirmed.
%\end{cor}

%% Remark: the condition $HN^s\gg p^s$ is necessary, but the other?

We are now ready for the study of distribution of $N_{\CB}(C)$ for small $\CB$. We fix an interval $\CJ\subseteq[0,p-1]$, and let $N=\abs{\CJ}$. For any $H>0$ (which may depends on $p$), let $\CB_x=(x,x+H]\times \CJ$. %Let $\CI\subseteq [0,p-1]$ be a given interval of length $\abs{\CI}\gg p^{\frac{1}{2}+\varepsilon}$.
From now on, we will assume the following condition.
\begin{equation}\label{cond1}
\text{For any given~} x, \text{~there is at most one~} y \text{~so that~} (x,y)\in C\cap\CJ.
\end{equation}
This is the same condition we imposed in \cite{MaZa11} when Zaharescu and the author study the distribution of hybrid exponential sums over curves.

% The restriction imposed by this condition is a serious one, but it holds for example when $C$ is a rational curve $y=f(x)$, or when $C$ is an hyperelliptic curve, $\CJ=(\alpha p,\beta p]$ and $0\leq \alpha<\beta\leq 1/2$.

Define
\begin{equation}\label{defMkH}
M_k(H)=\sum_{x=0}^{p-1}\left( N_{\CB_x}(C)-\frac{HN}{p} \right)^k
\end{equation}
to be the $k$-th moment of the number of points in $C\cap\CB_x$ about its mean. We also define $\mu_k(H,P)$ to be the $k$-th moment of a binomial random variable $X$ with parameter $H$ and $P$, i.e.
\begin{equation}\label{defmuk}
\mu_k(H,P):=E((X-HP)^k)=\sum_{h=1}^{H}\binom{H}{h}P^h(1-P)^{H-h}(h-HP)^k.
\end{equation}
We estimate the moment $M_k(H)$ using the binomial model with parameter $H$ and $N/p$.
\begin{theorem}\label{thm3}
Fix a positive integer $k$. Let $C$ be a curve satisfying the assumptions in Theorem \ref{thm1} and the additional condition \eqref{cond1}. Set $\CB_x$, $H$, $N$ as above, we have
\begin{equation*}
M_k(H)=p\nu(H,N/p) + O_k(d^{2k}H^k\sqrt{p}\log^k{p}).
\end{equation*}
\end{theorem}

For a fixed $k$, it is well-known (see Montgomery and Vaughan \cite{MoVa86}, Lemma 11) that
\begin{equation*}
\mu_k(H,P)\ll (HP)^{k/2}+HP
\end{equation*}
uniformly for $0\leq P\leq 1$ and $H=1,2,3,\ldots$. Therefore, Theorem \ref{thm3} immediately implies the following.
\begin{cor}\label{cor3}
Assumptions as in Theorem \ref{thm3}. For any fixed $k$, we have
\begin{equation*}
M_k(H) \ll_k p(HN/p)^{k/2}+HN/p+d^{2k}H^k\sqrt{p}\log^k{p}.
\end{equation*}
\end{cor}
\begin{remark}
%\begin{enumerate}
%\item Even when $N\sim c p$, $0 \leq c \leq 1$, the main term $p(HN/p)^{k/2}$ of corollary \ref{cor3} dominates only if $H$ is very small, namely when $H\ll_k p^{1/k}/(d^4\log^2{p})$.
%\item
For the case of curves, \cite[Theorem 2]{MaZa13} gives an upper bound for the second moment of $N_{\CB}(C)$ when $\CB$ is allowed to move freely on the plane. That theorem implies $M_2(H)\ll p\mu_2(H,N/p)$. Since $\mu_2(H,P)=HP(1-P)$, Theorem \ref{thm3} shows that \cite[Theorem 2]{MaZa13} has the correct main term, and therefore is best possible for the case of curves (with suitable $H$ and $N$).
%\end{enumerate}
\end{remark}

Let
\begin{equation*}
\nu_k =
\begin{cases}
1\cdot3\cdot\ldots\cdot(k-1) &, k \text{~even}, \\
0 &, k \text{~odd},
\end{cases}
\end{equation*}
then (see \cite[Lemma 10]{MoVa86})
\begin{equation*}
\mu_k(H,P)=(\nu_k+o(1))(HP(1-P))^{k/2}
\end{equation*}
as $HP(1-P)$ tends to infinity. From this and Theorem \ref{thm3} we obtain the following.
\begin{cor}\label{cor4}
For any fixed $k$, if $H=o\left(\frac{p^{1/2k}}{d^2\log{p}}\right)$ and $(HN/p)(1-N/p)\rightarrow \infty$ as $p$ tends to infinity, then
\begin{equation*}
M_k(H)=p(\nu_k+o(1))\left(\frac{HN}{p}\left(1-\frac{N}{p}\right)\right)^{k/2}.
\end{equation*}
In particular, when $N\sim c p$, $0 < c < 1$ and $\log{H}/\log{p}\rightarrow 0$ as $p$ tends to infinity, the distribution of $N_{\CB_x}(C)$ tends to a Gaussian distribution with mean $HN/p$ and variance $(HN/p)(1-N/p)$.
\end{cor}

\begin{remark}
If condition \eqref{cond1} does not hold, we may still have Gaussian distribution for the $N_{\CB_x}(C)$. For example, if $C$ is a hyperelliptic curve, and choose $\CJ$ to be the interval $(-\alpha p,\alpha p]$ for some $0 < \alpha < 1/2$, then generically one $x$-coordinate on the curve corresponds to two $y$-coordinates. From Corollary \ref{cor4}, we have Gaussian distribution for $\CJ_1=[0,\alpha p]$, and also for $\CJ_2=[-\alpha p,0]$, with the same mean and variance. After combining the two of them we will have Gaussian distribution for the whole interval $\CJ$.
\end{remark}

%% Estimation of gaps doable?

\section{Preliminary lemmas}

In this section we collect all the preliminary lemmas that will be used in the subsequent sections. The first lemma is the Weil bound for space curves. For a proof, see \cite[Theorem 2.1]{MaZa12}.
\begin{lemma}\label{lem21}
Let $C$ be an absolutely irreducible curve in the affine $r$-space $\mathbb{A}^r_p$ of degree $d>1$, which is not contained in any hyperplane. Let $\CB=\CI_1\times\ldots\times\CI_r$ be a box, then
\begin{equation*}
N_{\CB}(C)=p\cdot\frac{\vol(\CB)}{p^r}+O(d^2\sqrt{p}\log^t{p}),
\end{equation*}
where $t$ is the number of intervals $\CI_i$ that are not the full interval $[0,p-1]$.
\end{lemma}

The next lemma states that if we translate a set in $\mathbb{F}_p$ a small number of times, it will always reach a new element. This lemma allows us to show later that some curves are absolutely irreducible. %% May be improved I think?
\begin{lemma}\label{lem22}
Let $r\geq 2$, $x_1, \ldots, x_r\in\mathbb{F}_p$ be $r$ distinct elements. Suppose $\CM$ is a nonempty finite subset of the algebraic closure $\overline{\mathbb{F}}_p$ with $4\abs{\CM}<p^{\frac{1}{r}}$.
Then there exists a $j\in\{1,\ldots,r\}$ such that the translate $\CM+x_j$ is not contained in $\cup_{i\neq j}(\CM+x_i)$.
\end{lemma}
\begin{proof}
Suppose $(x_1,\ldots,x_r,\CM)$ provides a counterexample to the statement of the lemma. Then it is clear that for any nonzero $t\in\mathbb{F}_p$, the tuple $(tx_1,\ldots,tx_r,t\CM)$ is another counterexample.

By Minkowski's theorem on lattice points in a convex symmetric body, there exists a nonzero integer $t$ such that
\begin{equation*}
\begin{cases}
\abs{t} &\leq p-1 \\
\norm{\frac{tx_1}{p}} &\leq (p-1)^{-\frac{1}{r}} \\
&\vdots \\
\norm{\frac{tx_r}{p}} &\leq (p-1)^{-\frac{1}{r}}. \\
\end{cases}
\end{equation*}
Thus there are integers $y_j$ such that
\begin{equation}\label{condyj}
\begin{cases}
\abs{y_j} &\leq p(p-1)^{-\frac{1}{r}} \\
y_j &\equiv tx_j \pmod{p} \\
\end{cases}
\end{equation}
for any $j\in\{1,\ldots,r\}$, and $(y_1,\ldots,y_r,t\CM)$ provides a counterexample. Now let $j_0$ be such that $\abs{y_{j_0}}=\max_{1\leq j\leq r}\abs{y_j}$. Choose $\alpha\in t\CM$ and consider the set $\tilde{\CM}=t\CM\cap(\alpha+\mathbb{F}_p)$. Then $(y_1,\ldots,y_r,\tilde{\CM})$ will also be a counterexample.

Note that $\alpha+\mathbb{F}_p$ can be written as a union of at most $\abs{\CM}$ intervals (i.e. subsets of $\mathbb{F}_p$ consisting of consecutive integers or its translate in $\overline{\mathbb{F}}_p$) whose endpoints are in $\tilde{\CM}$. Let $\{\alpha+a,\alpha+a+1,\ldots,\alpha+b\}$ be the longest of these intervals. Then
\[\abs{b-a} \geq \frac{p}{|\tilde{\CM}|}\geq\frac{p}{\abs{\CM}}.\]
By this, \eqref{condyj} and the hypothesis $4\abs{\CM}<p^{\frac{1}{r}}$, we have
\[\abs{b-a}>4p^{1-\frac{1}{r}}>2\abs{y_{j_0}}.\]
Now if $y_{j_0}>0$, then $\alpha+a+y_{j_0}$ belongs to $\tilde{\CM}+y_{j_0}$ but does not belong to $\cup_{i\neq j_0}(\tilde{\CM}+y_i)$, while if $y_{j_0}>0$, then $\alpha+b+y_{j_0}$ belongs to $\tilde{\CM}+y_{j_0}$ but does not belong to $\cup_{i\neq j_0}(\tilde{\CM}+y_i)$. This contradicts the fact that $(y_1,\ldots,y_r,\tilde{\CM})$ is a counterexample, and completes our proof.
\end{proof}

Recall that the Stirling number of second kind, $S(r,t)$, is by definition the number of partition of a set of cardinality $r$ into exactly $t$ nonempty subsets. The proof of the following lemma can be found in \cite{MoVa86}.
\begin{lemma}\label{lem23}
Let $\mu_k(H,P)$ be defined by \eqref{defmuk}, then
\begin{equation*}
\mu_k(H,P) = \sum^k_{r=0}\binom{k}{r}(-HP)^{k-r}\left( \sum_{t=0}^r\binom{H}{t}S(r,t)t!P^t \right).
\end{equation*}
\end{lemma}

\section{Patterns of curves: Proof of Theorem \ref{thm1}}

Let $C$ be a plane curve given by the equation $f(x,y)=0$ and two vectors $\mathbf{a}=(a_1,\ldots,a_s), \mathbf{b}=(b_1,\ldots,b_s)$ so that $p\nmid a_i$ and $a_1^{-1}b_1,\ldots,a_s^{-1}b_s$ are all distinct modulo $p$, we define the $x$-\textit{shifted} curve $C_{\mathbf{a},\mathbf{b}}$ to be the space curve in the affine $(s+1)$-space with varaibles $x,y_1,\ldots,y_r$ and equations
\begin{equation}\label{eqnshift}
f(a_ix+b_i,y_i)=0, ~\forall 1\leq i\leq s.
\end{equation}
It is not difficult to see that $C_{\mathbf{a},\mathbf{b}}$ is indeed a curve, and its degree is less than or equal to $d^s$. Note that similar constructions appeared in \cite{MaZa11b, MaZa11, MaZa12}.

It is clear from the defining equations \eqref{eqnshift} that a point on $C_{\mathbf{a},\mathbf{b}}$ corresponds to an $(\mathbf{a},\mathbf{b})$-pattern on $C$, i.e.
\begin{equation*}
P_{\mathbf{a},\mathbf{b}}(C,\CI,\CJ) = N_{\CB}(C_{\mathbf{a},\mathbf{b}}),
\end{equation*}
where $\CB=\CI\times(\CJ)^s$. We want to show that $C_{\mathbf{a},\mathbf{b}}$ is absolutely irreducible. Currently we are not able to prove this for all curves $C$, but we are able to show the irreducibility for the class of curves so that the projection $\pi$ defined by \eqref{eqnpi} has a completely ramified point.
\begin{prop}\label{prop31}
If $C$ satisfies the assumptions in Theorem \ref{thm1}, then $C_{\mathbf{a},\mathbf{b}}$ is absolutely irreducible.
\end{prop}
\begin{proof}
For $1\leq j\leq s$ we define $C_j$ to be the curve given by the first $j$ equations in \eqref{eqnshift}, i.e.
\begin{equation*}
f(a_ix+b_i,y_i)=0, ~\forall 1\leq i\leq j.
\end{equation*}
We have a chain of coverings of curves,
\begin{equation*}
C_{\mathbf{a},\mathbf{b}}=C_s\rightarrow C_{s-1}\rightarrow\ldots\rightarrow C_1\cong C,
\end{equation*}
where each arrow represent a projection $\pi_i$ given by $(x,y_1,\ldots,y_i)\mapsto(x,y_1,\ldots,y_{i-1})$. Let $S\subseteq\overline{\mathbb{F}}_p$ be the set of completely ramified points for the map $\pi:C\rightarrow\mathbb{A}^1$. Since all the $x_i=b_ia_i^{-1}$ are distinct, we can apply Lemma \ref{lem22} with $x_i=b_ia_i^{-1}$ to conclude that there are new completely ramified points in each level of the above chain of coverings. Since $C$ is absolutely irreducible, this shows that $C_{\mathbf{a},\mathbf{b}}$ is also absolutely irreducible.
\end{proof}

We are now ready to prove Theorem \ref{thm1}. By Proposition \ref{prop31}, if $C$ satisfies the assumptions in the theorem, then $C_{\mathbf{a},\mathbf{b}}$ is absolutely irreducible in $\mathbb{A}^{s+1}$. Theorem \ref{thm1} now follows easily from Lemma \ref{lem21}.

\section{Estimation of $M_k(H)$: Proof of Theorem \ref{thm3}}

Using the binomial theorem to expand the right hand side of \eqref{defMkH}, we obtain
\begin{align*}
M_k(H)&=\sum_{x=0}^{p-1}\sum_{r=0}^k\binom{k}{r}N_{\CB_x}(C)^r\left(-\frac{HN}{p}\right)^{k-r} \\
&= \sum_{r=0}^k\binom{k}{r}\left(-\frac{HN}{p}\right)^{k-r}\sum_{x=0}^{p-1}N_{\CB_x}(C)^r.
\end{align*}
Here we make the convention that if $r=0$, $N_{\CB_x}(C)^r=1$ even when $N_{\CB_x}(C)=0$. Define
\begin{equation*}
S_r(H)=\sum_{x=0}^{p-1}N_{\CB_x}(C)^r.
\end{equation*}
Clearly $S_0(H)=p$ (by our convention). For $r\geq 1$, we have
\begin{equation}\label{eqn41}
S_r(H)=\sum_{x=0}^{p-1}\sum_{(x_1,y_1)\in C\cap\CB_x}\cdots\sum_{(x_r,y_r)\in C\cap\CB_x} 1 =\sum_{x=0}^{p-1}\sum_{\substack{(x_i,y_i)\in C, y_i\in\CJ \\ \{x_1,\ldots,x_r\}\subseteq(x,x+H]}} 1.
\end{equation}
For each $1\leq i\leq r$, let $x_i=x+a_i$, and let $A$ be the set of distinct $a_i$'s. Set $\abs{A}=t$. We have $A\subseteq\{1,2,\ldots,H\}$. From the definition of the Stirling number of second kind, we see that for any given $A$, the number of sets with $\{x_1,\ldots,x_r\}=A$ is $S(r,t)t!$. Grouping the terms in \eqref{eqn41} according to different values of $t$, we obtain
\begin{equation}\label{eqn42}
S_r(H)=\sum_{t=1}^rS(r,t)t!\sum_{\substack{\abs{A}=t\\ A\subseteq[1,H]}}\sum_{x=0}^{p-1}\sum_{\substack{(x+b_i,y_i)\in C, 1\leq i\leq r \\ y_i\in \CJ}} 1.
\end{equation}
By condition \eqref{cond1}, the inner sum
\begin{equation*}
\sum_{x=0}^{p-1}\sum_{\substack{(x+b_i,y_i)\in C, 1\leq i\leq r \\ y_i\in \CJ}} 1
\end{equation*}
is the number of $(\mathbf{a},\mathbf{b})$-pattern of $C$ with $\mathbf{a}=(1,1,\ldots,1)$, $\mathbf{b}$ is any $t$-tuple ordering of the set $A$, and all $y$ coordinates lie in $\CJ$. By Theorem \ref{thm1}, this sum is
\begin{equation*}
\sum_{x=0}^{p-1}\sum_{\substack{(x+b_i,y_i)\in C, 1\leq i\leq r \\ y_i\in \CJ}} 1 = p\cdot\frac{N^t}{p^{t}}+O(d^{2t}\sqrt{p}\log^t{p}).
\end{equation*}
Put this into \eqref{eqn42} yields
\begin{align*}
S_r(H) &= \sum_{t=1}^rS(r,t)t!\sum_{\substack{\abs{A}=t\\ A\subseteq[1,H]}}\left(p\cdot\frac{N^t}{p^{t}}+O(d^{2t}\sqrt{p}\log^t{p})\right) \\
&= p\sum_{t=1}^rS(r,t)t!\binom{H}{t}\left(\frac{N}{p}\right)^t +O\left(\sum_{t=1}^rS(r,t)t!\binom{H}{t}d^{2t}\sqrt{p}\log^t{p}\right).
\end{align*}
Therefore,
\begin{equation*}
M_k(H)=p\sum_{r=0}^k\binom{k}{r}\left(-\frac{HN}{p}\right)^{k-r}\sum_{t=1}^rS(r,t)t!\binom{H}{t}\left(\frac{N}{p}\right)^t + O_k(d^{2k}H^k\sqrt{p}\log^k{p}).
\end{equation*}
We can insert the terms with $t=0$ without altering the sum since $S(r,0)=0$ for any $r\geq 1$ (and for $r=0$ the inner sum is understood to be zero), thus we may apply Lemma \ref{lem23} with $P=N/p$ to conclude that
\begin{equation*}
M_k(H)=p\nu(H,N/p) + O_k(d^{2k}H^k\sqrt{p}\log^k{p}).
\end{equation*}
This completes the proof of Theorem \ref{thm3}.

%\bibliography{pi13}
%\bibliographystyle{amsplain}

\end{document}